\documentclass{article}
\usepackage[utf8]{inputenc}
\usepackage[english]{babel}
\usepackage{polski}
\usepackage[colorlinks=true, citecolor=blue, linkcolor=blue, linktocpage=true]{hyperref}
\newcommand{\are}{\preccurlyeq}
\newcommand{\s}{F}

\newcommand{\ing}{P}
\newcommand{\pr}{PP}
\newcommand{\UInt}{\mathsf{U}}
\newcommand{\I}{\mathsf{I}}
\newcommand{\supr}{Mub}
\usepackage{amsmath}		
\usepackage{amsthm}
\usepackage{amsfonts}				
\usepackage{amssymb}	
\usepackage{amsbsy}
\usepackage{latexsym}
\usepackage{hyperref}

\usepackage{geometry}
\newgeometry{tmargin=2cm, bmargin=3cm, lmargin=4.2cm, rmargin=4.2cm}

\let\oldbibliography\thebibliography
\renewcommand{\thebibliography}[1]{%
 \oldbibliography{#1}%
\setlength{\itemsep}{2pt}%
}

\let\OLDthebibliography\thebibliography
\renewcommand\thebibliography[1]{
  \OLDthebibliography{#1}
  \setlength{\parskip}{0.7pt}
  \setlength{\itemsep}{0.7pt plus 0.7ex}
 }

\usepackage{amsmath}
\usepackage{hyperref}
\usepackage{breqn}

\newtheorem{theorem}{Theorem}
\newtheorem{lemma}{Lemma}

\title{Classical Mereology is Axiomatizable Using \\
Primitive Fusion in Two-Sorted Logic}
\author{Marcin Łyczak}
\date{\small{The work contains my axiomatization of \textit{Classical Mereology} with a primitive notion of \textit{mereological fusion}. I presented this work, along with all the necessary proofs, to Achille Varzi in a series of seminars in November 2022 during my scholarship at Columbia University, financed by the University of Warmia and Mazury (POWR.03.05.00-00-Z310/17).}}
\begin{document}
\maketitle
\begin{abstract}
\textit{Mereological fusion}, also known as \textit{composition} and \textit{sum}, was originally used by me as a primitive notion to axiomatize \textit{Extensional Mereology} with \textit{atoms} in \cite{Ly22}. Here, I extend this idea to axiomatize \textit{General Extensional Mereology}, also called \textit{Classical Mereology}, which is neutral regarding the existence of atoms. I give a proof that classical mereology is axiomatizable using primitive mereological fusion within the framework of two-sorted logic, as I announced in \cite{Ly22}. The use of the primitive notion of fusion instead of primitive notion of  \textit{part} was considered by G. Leibniz \cite[50]{CoVa21}, S. Leśniewski \cite[CCLXIV, CCLXIV]{Le92}, K. Fine \cite{Fi10}, J. Ketland, and T. Schindler \cite{KeSh16}, and S. Kleishmid \cite{Kl17}. C. Lejewski formulated mereology using a single axiom with primitive mereological sum \cite[222]{Sb84}. However, Lejewski's theory is formulated in non-classical language of Leśniewski's ontology and contains complicated quantification over function symbols, including quantification on \textit{part of} and \textit{sum of}. Lejewski's approach is not expressible in modern mereologies. The presented theory is the first contemporary axiomatization of classical mereology in two-sorted logic with primitive mereological fusion.
\end{abstract}
\section{Two-sorted logic framework for both mereologies}
The language of logic consists of variables of two sorts. The first for individuals ($x, y, z, ...$), and the second for pluralities ($zz, yy, zz, ...$). The predicate $\prec$, read as \textit{is one of}, is applied to individuals and pluralities, respectively. We use first-order identity $=$ between variables of the first sort. We use notation from \cite{FlLi21} because it is often used in mereological considerations that cannot be expressed in first-order logic.
The considered theories are expressed within a two-sorted logic framework (see \cite{MaAr22}) that includes axioms for first-order identity, for the first sort, and the unrestricted comprehension schema \cite[285]{Ma05}. Thus, we allow for the existence of an empty name in the case of the second sort. We take
the following instantiations of the comprehension schema, and definitions:
\begin{align*}
\tag{$\I$}\label{I} x\prec \I y &\leftrightarrow x=y,\\
\tag{$\cup$}\label{cup} x\prec yy\cup zz 
 &\leftrightarrow  x\prec yy \lor x\prec zz,\\
\tag{$\cap$}\label{cap} x\prec yy\cap 
 zz &\leftrightarrow  x\prec yy \land x\prec zz,\\
   \tag{$\are$}\label{are} xx\are yy & \leftrightarrow \forall z (z\prec xx \to z \prec yy),\\
\tag{$\approx$}\label{=}  xx\approx yy & \leftrightarrow  xx\are yy \land yy\are x.
\end{align*}
\section{Axiomatisation of Classical Mereology with primitive fusion}
The only primitive notion of the theory is a predicate $F$ representing \textit{mereological fusion}, which we also call \textit{mereological composition}. We use the defined function symbol $\UInt$ read \textit{components of}. The second sort term $\UInt xx$ will be defined using the composition as all objects that form wholes that are $xx$.
We take the following specific axioms for $F$:
\begin{gather}
 \tag{$\exists_{\s}$}\label{esum} \exists x (x\prec zz)\to \exists y F_{zz}y,\\
 \tag{$\approx_{F}$}\label{F=}  F_{xx}z\land xx\approx yy \to F_{yy}z,\\
   \tag{$\mathsf{ext}_{\s}$}\label{ext} \s_{zz}x \land \s_{yy}x \land \s_{uu\cup zz}v  \to  \s_{uu\cup yy}v,\\
      \tag{$\mathtt{id}_{\s}$}\label{fs} \s_{\I{y}}x \to x=y,\\
   \tag{$\mathtt{comp}_{\s}$} \label{comp}  \s_{zz\cup\I x}y\land \s_{zz}y \to \exists _{vv\are \UInt zz} (\s_{vv}x\land \exists z(z\prec vv)),\\
  \tag{$\mathtt{wsp}_{\s}$}\label{WSP}  \notag \s_{\I x \cup \I y}y \land x\not = y \to \exists_{z \prec \UInt \I y} \neg \exists u (\s_{\UInt \I x\cap \UInt \I z }u),
\end{gather}
where $\UInt zz$ is defined by an instantiation of the comprehension schema as
\begin{equation}
\tag{$\mathtt{df}\UInt_{F}$}\label{defUcpn} x\prec \UInt zz  \leftrightarrow  \exists_{z\prec zz}\exists yy(\s_{yy}z\land x \prec yy).
\end{equation}
We briefly comment on axiomatization with the primitive predicate $F$.
\eqref{esum} expresses Leśniewski's idea that there is a mereological fusion of any objects. \eqref{F=} describes a indiscernibility of identicals by $F$. \eqref{ext} is a version of extensionality, which states that if some $x$ is a fusion of $zz$ and fusion of $yy$ ($zz$ and $yy$ are possibly different), then both $zz$ and $yy$ can be substituted for each other when creating various fusions, as long as at least one of them is required. \eqref{fs} expresses the idea that the fusion of all objects identical to $y$ does not create a new individual; in other words, the fusion is identical to $y$. In mereology with part as primitive, \eqref{fs} is an obvious consequence of basic properties of being a part, and definition of fusion, which was noted already by Leśniewski.
 \eqref{comp} describes the constructiveness of composition: if $y$ is the fusion of $zz$ and $x$, but $x$ is unnecessary, then $x$ can be composed of some $yy$ that forms some 
wholes belonging to $zz$. An informal version of \eqref{WSP} was considered by S. Kleinschmidt in \cite[13]{Kl17}. \eqref{WSP} is closely connected with the so-called \textit{weak supplementation principle}, which we will introduce shortly.\smallskip\\
We denote the set of mereology with primitive fusion as
$$\mathsf{GEM}^{\mathsf{pl}}_F= \eqref{esum}+\eqref{F=}+ \eqref{fs}+\eqref{ext}+\eqref{comp}+\eqref{WSP}.$$
In $\mathsf{GEM}^{\mathsf{pl}}_{F}$, we define the notion of being a part as
\begin{equation}
\tag{$\mathtt{df}\ing_{\s}$}\label{dfing} \ing xy \leftrightarrow \exists zz(F_{zz} y \land x\prec zz).
\end{equation}
The counterpart of this definition was a theorem in the original Leśniewski's mereology. Leśniewski proved it and noted that this shows that the mereological fusion may be a primitive notion of mereology, instead of the primitive notion of mereological part \cite{Le92}:
\begin{quote} 
\begin{small}
THEOREM CCLXIV. $P$ is an ingredient of object $Q$, when and only when, for some $a$, $Q$ is the sum of objects $a$, and $P$ is $a$.\smallskip\\
Theorems $[...]$, CCLXIV, $[...]$ show that instead of such terms as `part' and `ingredient' I could choose as the fundamental term of my `general theory of sets', any of the terms: $[...]$, `class', $[...]$.
\end{small}
\end{quote}
\section{Axiomatization of Classical Mereology with primitive part}
We use nomenclature from \cite{Va16}. Axioms of the so-called \textit{Classical Mereology} (see \cite{Ho08}), called also \textit{General Extensional Mereology}, we express in the 
two-sorted logic are:
\begin{gather}
    \tag{$\mathtt{as}_{\pr}$} \label{aspr}  \pr xy\to \neg  \pr y x,\\
\tag{$\mathtt{trans}_{\pr}$} \label{transpr} \pr xy \land  \pr yz \to  \pr xz,\\
   \tag{$\exists_{F}$}\label{eFP} \exists x (x\prec zz)\to \exists y F_{zz}y,\\
      \tag{$\mathtt{fun}_{F}$}\label{fsum} F_{zz}x \land F_{zz}y \to x=y,
\end{gather}
where
\begin{gather}
   \tag{$\mathtt{df}F_{P}$} \label{defsum} F_{zz}x  \leftrightarrow  \forall_{y\prec zz} (\ing zx) \land \forall y(\ing y x \to \exists_{v\prec zz} (O v y )),\\
    \tag{$\mathtt{df}\ing_{\pr}$}\label{ing/pr} \ing xy \leftrightarrow\pr xy \lor x=y,\\
    \tag{$\mathtt{df}O$}\label{O} O xy\leftrightarrow \exists z(\ing z x\land \ing zy).
    \end{gather}
To characterize mereology it is often chosen primitive notion of \textit{being a part}, in an inclusive sense, with different axioms:
\begin{gather}
    \tag{$\mathtt{ref}_{\ing}$} \label{ring} \ing xx,\\
    \tag{$\mathtt{antis}_{\ing}$} \label{antising} \ing xy \land \ing yx \to x=y,\\
\tag{$\mathtt{trans}_{\ing}$} \label{transing} \ing xy \land \ing yz \to \ing x z,
\end{gather}
and $\pr$ is defined as
\begin{equation}
\tag{$\mathtt{df}\pr_{\ing}$} \label{defpring}
   \pr xy \leftrightarrow\ing xy \land x\not =y.
\end{equation}
It as a well-known fact that 
$$\eqref{ring}+\eqref{antising}+
\eqref{transing}+\eqref{defpring}=\eqref{aspr}+\eqref{transpr} +\eqref{ing/pr}.$$
Thus, from a formal perspective, the choice of the two primitives doesn't make much difference. However, for convenience, $P$ is often preferred.\smallskip\\
We denote the set of theses of classical mereology by
$$\mathsf{GEM}^{\mathsf{pl}}_{P}=\eqref{ring}+\eqref{antising}+\eqref{transing}+\eqref{esum}+\eqref{fsum}.$$
\section{Definitional equivalence of  \texorpdfstring{$\mathsf{GEM}^{\mathsf{pl}}_{P}$}{LG} and \texorpdfstring{$\mathsf{GEM}^{\mathsf{pl}}_{F}$}{LG}  }
\begin{lemma}
In $\mathsf{GEM}^{\mathsf{pl}}_{F}$ the following formulas are provable:
\begin{gather}
    \tag{$F_{\I x}x$}\label{FIx} F_{\I x}x,\\
    \tag{$P_{F}2$}\label{PP} \ing xy \leftrightarrow F_{\I x\cup \I y}y.
\end{gather}
\end{lemma}
\begin{proof}
To prove \eqref{FIx} we use the fact that $x\prec \I x$ and so from axiom \eqref{esum} we have $\exists y F_{\I x} y$. Thus, using classical logic and \eqref{fs}, we obtain $\exists y (F_{\I x} y \land x=y)$, and so $F_{\I x}x$. 
To prove `$\leftarrow$' of \eqref{PP} we assume $F_{\I x\cup \I y}y$ and then we just use the fact that $x\prec \I x\cup \I y$ and definition of \eqref{dfing}. 
To prove `$\to$' of \eqref{PP} we assume $\ing xy$ and from definition \eqref{dfing} we have $F_{aa}y\land x\prec aa$, for some $aa$. Having $x\prec aa$ we obtain from logic  $aa\approx \I x \cup aa$, so using $F_{aa}y$ and \eqref{F=} we get $F_{\I x\cup aa}y$. So we have $F_{\I x\cup aa}y$, $F_{aa}y$, and $F_{\I y}y$ from \eqref{FIx}, thus we use \eqref{ext} and obtain $F_{\I x\cup \I y}y$.
\end{proof}
\begin{lemma} \label{lemmartrant}
In $\mathsf{GEM}^{\mathsf{pl}}_{F}
+\eqref{dfing}$ formulas \eqref{ring}, \eqref{antising}, and \eqref{transing} are provable.
\end{lemma}
\begin{proof}
To prove \eqref{ring} we just use \eqref{FIx} and by \eqref{dfing} and $x\prec \I x$ we obtain $Px x$. To prove  \eqref{antising} we assume $\ing x y$ and $ \ing y x$. From this using \eqref{PP} and \eqref{F=} we obtain $F_{\I x \cup \I y}x$ and $F_{\I x\cup \I y}y$. From the latter with the use of logical thesis $\I x\cup \I y\approx\I x\cup \I x\cup \I y$  and \eqref{F=} we obtain $F_{\I x\cup \I x\cup \I y}y$. Now we take \eqref{ext} with: $\I x\cup \I y/zz, \I x/yy, \I x/uu, y/v$ and we obtain as $\mathsf{GEM}^{\mathsf{pl}}_{F}$ thesis the following
$$(\star)\colon F_{\I x\cup \I y}x \land F_{\I x}x \land F_{\I x\cup \I x\cup \I y} y\to F_{\I x\cup\I 
 x}y.$$
We have all predecessors of ($\star$) so we get $F_{\I x \cup \I x}y$, i.e $F_{\I x}y$ by $\I x \cup \I x\approx \I x$ and \eqref{F=}. Having $F_{\I x}y$ we just use \eqref{fs} and have $x=y$. To prove \eqref{transing} we assume $\ing xy\land \ing y z$. Using \eqref{PP} we obtain $F_{\I x\cup\I y}y\land F_{\I y\cup \I z}z$.
Now we use \eqref{ext} with: $\I y/zz$, $\I x \cup \I y/yy$, $\I z/uu$,  $y/x$, $z/v$ and we obtain as $\mathsf{GEM}^{\mathsf{pl}}_{F}$ thesis the following
$$(\star \star)\colon F_{\I y}y  \land F_{\I x\cup \I y}y \land F_{\I z\cup \I y} z \to F_{\I z\cup \I x \cup \I y}z.$$ $F_{\I x\cup\I y}y\land F_{\I y\cup \I z}z$ we already have and $F_{\I y}y$ we have as earlier from \eqref{FIx}, so we have all predecessors of $(\star \star)$, and thus we obtain $F_{\I z\cup \I x \cup \I y}z$, so $\ing xz$ directly from \eqref{dfing} and $ z\prec \I z\cup \I x \cup \I y$.
\end{proof}
Next we show uniqueness of fusion in $\mathsf{GEM}^{\mathsf{pl}}_{F}$. 
\begin{lemma}
In $\mathsf{GEM}^{\mathsf{pl}}_{F}+\eqref{dfing}$ formula \eqref{fsum} is provable.
\end{lemma}
\begin{proof}
We assume $F_{zz}x$ and $F_{zz}y$. From $F_{zz}x$, $F_{\I x}x$, and $F_{zz\cup zz}y$ (taken from $F_{zz}y$, $zz\approx zz\cup zz$ and \eqref{F=}) using \eqref{ext} we have $F_{zz\cup \I x}y$ thus $\ing xy$ by \eqref{dfing}. Analogically from  $F_{zz}y$, $F_{\I y}y$, and $F_{zz\cup zz}x$ using \eqref{ext} we have $F_{zz\cup \I y}x$ thus $\ing yx$ by \eqref{dfing}. So $\ing xy$ and $\ing yx$. Therefore, using proved \eqref{antising} we get $x=y$.  
\end{proof} 
Now we show that fusion axiomatized in $\mathsf{GEM}^{\mathsf{pl}}_{F}$ is fusion in Leśniewski's sense.
\begin{lemma}\label{cltosum}
In $\mathsf{GEM}^{\mathsf{pl}}_{F}+\eqref{dfing}$ formula $F_{zz}x \to  \forall_{y\prec zz} (\ing yx) \land \forall y(\ing yx \to  \exists_{v\prec zz}(O  v y))$ is provable.
 \end{lemma}
 \begin{proof}
We assume $F_{zz}x $. From this and \eqref{dfing} we directly have $\forall_{y\prec zz} \ing yx$. We assume $\ing yx$, for any $y$ and by \eqref{PP} we obtain $F_{\I y\cup \I x}x$. From this, \eqref{FIx}, and assumed $F_{zz}x$ we obtain $F_{zz\cup\I y}x$ by \eqref{ext}. Thus we have $F_{zz}x $ and $F_{zz\cup\I y}x$. We use \eqref{comp} and there is $aa\are \UInt zz$ such that $F_{aa}y$ and $a\prec aa$ for some $a$. From $aa\are \UInt zz$ and $a\prec aa$ we have $a\prec \UInt zz$ thus applying \eqref{defUcpn} for some $v\prec zz$ and $yy$ we have $F_{yy}v\land a\prec yy$. So using the latter and \eqref{dfing}, we have $\ing a v$. We also know that $F_{aa}y$ and $a\prec aa$, so $Pay$. Thus applying \eqref{O} we obtain $O v y$, and we know that $v\prec zz$, which ends the proof. 
\end{proof}
To prove the converse implication we need the following Lemma
\begin{lemma}\label{FUIx}
In $\mathsf{GEM}^{\mathsf{pl}}_{F}+\eqref{dfing}$ formula $F_{\UInt \I x}x$ is provable.
 \end{lemma}
 \begin{proof}
First, we show that $F_{\UInt \I x}y\to x=y$ is a thesis of  $\mathsf{GEM}^{\mathsf{pl}}_{F}$. We assume that $F_{\UInt \I x}y$ and $x\not =y$. From assumed  $F_{\UInt \I x}y$, and the fact that $\UInt \I x\approx \UInt \I x\cup \I x$ using \eqref{F=} we obtain that $F_{\UInt \I x\cup \I x}y$. The latter with $F_{\UInt \I x}y$, and 
$F_{\I y}y$ by \eqref{ext} gives $F_{\I x \cup \I y}y$. Having $F_{\I x \cup \I y}y$ and assumed $x\not = y$ we use \eqref{WSP} and we obtain that there is $a$ such that $a \prec \UInt \I y \land \neg \exists u (\s_{\UInt \I x\cap \UInt \I a }u)$. $a \prec \UInt \I y$ from definitions \eqref{defUcpn} and \eqref{I} gives that for some $aa$ we have $F_{aa}y\land a\prec aa$, so with the use of $aa\approx aa\cup \I a$ and \eqref{F=} we get $F_{aa\cup\I a}y$. Having $F_{\UInt \I x}y$ and $F_{aa}y$ and $F_{aa\cup\I a}y$ we use \eqref{ext} and obtain $F_{\UInt \I x\cup \I a}y$. So we have $F_{\UInt \I x\cup \I a}y$ and $F_{\UInt \I x}y$ thus we apply \eqref{comp} and we obtain that there are $bb$ and $b$ such that $b\prec bb$ and $bb\are \UInt \UInt \I x$ and $F_{bb}a$. Thus we have $b\prec \UInt \UInt \I x$ and $(i)\colon\ing ba$ from \eqref{dfing}, $F_{bb}a$, and $b\prec bb$. From $b\prec \UInt \UInt \I x$ we have that there are $cc$ and $c$ such that $b\prec cc$, $c\prec \UInt \I x$, and $F_{cc}c$, so $(ii)\colon Pbc$ using \eqref{dfing}. Now from $c\prec \UInt \I x$ we have that there is $dd$ and such that  $F_{dd}x\land c\prec dd$, so $(iii)\colon Pcx$ by \eqref{dfing}. Thus from $(ii)$ and $(iii)$ using \eqref{transing} we obtain $Pbx$. So, we have $Pbx$ and $Pba$ from (i) and using definitions \eqref{dfing} and \eqref{defUcpn} we get $b\prec \UInt \I x\land b\prec \UInt \I a$, thus $b\prec \UInt \I x\cap \UInt \I a$. Now we can use the axiom of the existence of fusion \eqref{esum} which gives $\exists u (\s_{\UInt \I x\cap \UInt \I a }u)$ which is false what we showed earlier, so we finally have $\forall y (F_{\UInt \I x}y\to x=y)$. We know that $x\prec \UInt \I x$ so we again use \eqref{esum} and we have $F_{\UInt \I x}d$ for some $d$, but then we use already proved thesis and get $x=d$, so finally $F_{\UInt \I x}x$.
 \end{proof}
 \begin{lemma}\label{sumtocl}
In $\mathsf{GEM}^{\mathsf{pl}}_{F}+\eqref{dfing}$ formula $  \forall_{y\prec zz} (\ing yx) \land \forall y(\ing yx \to  \exists_{v\prec zz}(O  v y))\to F_{zz}x$ is provable.
\end{lemma}
\begin{proof}
We assume $\forall_{y\prec zz} (\ing yx) \land \forall y(\ing yx \to  \exists_{v\prec zz}(O  v y))$ from $\forall_{y\prec zz} (\ing yx)$ and definitions \eqref{dfing} and \eqref{defUcpn} we obtain $zz\are \UInt \I x$, thus $\UInt \I x\approx  zz\cup \UInt \I x$. From Lemma \ref{FUIx} we know that $F_{ \UInt \I x}x$, thus using  $\UInt \I x\approx  zz\cup \UInt \I x$ and \eqref{F=} we obtain $F_{zz\cup \UInt \I x}x$.
From $\forall y(\ing yx \to  \exists_{v\prec zz}(O  v y))$ and proved reflexivity of $P$ \eqref{ring} we obtain that $\exists v (v\prec zz)$. From the latter and axiom of the existence of fusion \eqref{esum} we have that $F_{zz}a$ for some $a$. So, we have $F_{zz}a$ and $F_{\I a}a$ and $F_{zz\cup \UInt \I x}x$ thus using \eqref{ext} we obtain 
$F_{\I a\cup \UInt \I x}x$ and we have $F_{\UInt \I x}x$ and $F_{\I x}x$ so again using \eqref{ext} we have $F_{\I a\cup \I x}x$. Now, we additionally assume indirectly that $x\not = a$. From \eqref{WSP} for some $b$ we obtain that $b\prec \UInt \I x \land \neg \exists u (\s_{\UInt \I b\cap \UInt \I a }u)$. From $b\prec \UInt \I x$ and definition \eqref{defUcpn} we obtain $\exists yy (F_{yy}x\land b\prec yy)$, i.e. $Pbx$ by \eqref{dfing}. Now we take assumption $\forall y(\ing yx \to  \exists_{v\prec zz}(O  v y))$ with $b/y$ and using $Pbx$ for some $c$ we obtain $c\prec zz \land O c b$. From $O  c b$ we obtain that for some $d$ we have $Pdc\land Pdb$. From $Pdc$ and $c\prec zz$ using \eqref{dfing} we have $d\prec \UInt zz$. Having $d\prec \UInt zz$ with the use of $F_{zz}a$ and \eqref{ext} we get $d\prec \UInt \I a$. From $Pdb$ we obtain $d\prec \UInt \I b$ by definitions \eqref{dfing} and \eqref{defUcpn}. Thus, we have  $d\prec \UInt \I b$ and $d\prec \UInt \I a$ and next $d\prec \UInt \I b \cap \UInt \I a$. The latter using \eqref{esum} gives $\exists u (\s_{\UInt \I b\cap \UInt \I a }u)$, which is false, so finally we obtain $x=a$. We know that $F_{zz}a$ so using $x=a$ we finally obtain $F_{zz}x$, what we wanted to prove.
\end{proof}
Thus using Lemmas \ref{lemmartrant}-\ref{cltosum} and \ref{sumtocl} we obtain
\begin{theorem}\label{PsubF}
$\mathsf{GEM}^{\mathsf{pl}}_{P}\subseteq  \mathsf{GEM}^{\mathsf{pl}}_{F}+\eqref{dfing}$
\end{theorem}
Now we show converse dependency. 
\begin{lemma}
    In $\mathsf{GEM}^{\mathsf{pl}}_{P}$ the following formulas are provable:
\begin{gather}
\tag{$\mathtt{WSP}$} \label{wsp}  PP xy \to  \exists z (PP zy \land \neg O zx),\\
\tag{$F_{P}=\supr$} \label{sum=sup} F_{zz}x \leftrightarrow \supr_{zz}x,
\end{gather}
where $\supr$ is defined as
\begin{equation}
    \tag{$\mathtt{df Mub}$} \label{defsup} Mub_{zz}x \leftrightarrow \exists y (y\prec zz)\land \forall_{y\prec zz} (\ing yx) \land \forall y (\forall_{v\prec zz} (\ing vy)\to \ing xy).
\end{equation}
\end{lemma}
Weak supplementation principle  $\eqref{WSP}$ was analyzed, and used as anaxiom, by P. Simons in his minimal extensional mereology \cite[28]{Si87}, but its stronger versions were also used as axioms by F. Drewnowski and Leśniewski \cite{Sw}. Analog of \eqref{sum=sup} was proved in Leśniewski's mereology by A. Tarski \cite[327–328]{Le92}.
\begin{lemma}\label{lid}
In $\mathsf{GEM}^{\mathsf{pl}}_{P}$ formula \eqref{fs} is provable.
\end{lemma}
\begin{proof}
We assume $\s_{\I{y}}x$. We get $\ing yx$ and using \eqref{sum=sup} we take $\forall_{v\prec \I y} (\ing vy)\to \ing xy)$. We take $y/v$ and using \eqref{ring} we obtain $\ing xy$. Thus we have  $\ing yx$ and $\ing yx$ so using \eqref{antising} we have $x=y$.
\label{extlms}
\end{proof}
\begin{lemma}
In $\mathsf{GEM}^{\mathsf{pl}}_{P}$ principle of fusion extensionality \eqref{ext} is provable. 
\end{lemma}
\begin{proof}
We assume (i): $\s_{zz}x$ (ii): $\s_{yy}x$, and (iii): $\s_{uu\cup zz}v$. From (i) and \eqref{sum=sup} we have $\forall_{z\prec zz}(\ing zv)\to\ing xv$. Predecessor, we obtain from (iii) and \eqref{defsum}, so $\ing xv$.
 Now from (ii) and \eqref{defsum} we have $\forall_{z\prec yy}(\ing zx)$, thus by \eqref{transing} and obtained $\ing xv$ we have $\forall_{z\prec yy}(\ing zv)$. From the latter using (iii) and \eqref{defsum} we get $\forall_{z\prec uu\cup yy}(\ing zv)$. Now for any $u$ such that $\ing uv$ we assume indirectly that $\forall_{y\prec uu\cup yy}(\neg O u y)$, i.e. ($\star$): $\forall_{y\prec uu}(\neg O u y)$ and $(\star\star)\colon \forall_{y\prec yy}(\neg O uy)$. From $\ing uv$, (iii), \eqref{defsum}, and ($\star$) we obtain $O ua$ for some $a\prec zz$, thus $Pax$, because $F_{zz}x$. Now using $\ing ax$ and $O ua$ we have $O ux$, i.e there is $b$ such that $\ing bu \land \ing bx$. From $\ing bx$, (ii) and \eqref{defsum} we have $\exists_{y\prec yy}(O b y)$, and so using $\ing bu$ we get $\exists_{y\prec yy}(O uy)$ which contradicts with ($\star\star$). So finally we have $\forall_{z\prec uu\cup yy}(\ing zv)\land \forall u (\ing uv\to  \exists_{z\prec uu\cup yy}(O zu))$ and $F_{yy\cup uu}v$ by \eqref{defsum}. 
\end{proof}
In $\mathsf{GEM}^{\mathsf{pl}}_{P}$ we define symbol $\UInt$ in the following way
\begin{equation}\tag{$\mathtt{df}\mathsf{\UInt}_{P}$} \label{UP}x\prec \UInt zz\leftrightarrow \exists_{y\prec zz} Pxy.
\end{equation}
\begin{lemma} \label{LMScomp}
In $\mathsf{GEM}^{\mathsf{pl}}_{P}+\eqref{UP}$ principle of composition \eqref{comp} is provable.
\end{lemma}
\begin{proof}
Proof indirect. We assume $F_{zz\cup\I x}y \land F_{zz}y $ and $\neg \exists_{yy\are  \UInt zz} (F_{yy}x)$. From $F_{zz\cup\I x}y$ and the definition of fusion \eqref{defsum} we obtain $P xy$. Let 
$u\prec zz^{*}\leftrightarrow \ing ux \land u\prec \UInt zz$.
Full comprehension schema guarantees that such a plurality exists. From $F_{zz}y$, $Pxy$ and \eqref{defsum}, we know that $\exists_{ z\prec zz} Ozx$ thus $\exists z 
\exists u(P u x\land Puz  \land z\prec zz)$, $ 
\exists u(P u x\land \exists z(Puz  \land z\prec zz))$ and next $
\exists u(Pux  \land u\prec \UInt zz)$ by \eqref{UP}, i.e. $\exists u(u\prec zz^{*})$ by definition od $zz^{*}$, thus  from \eqref{esum} for some $a$ we have $F_{zz^{*}}a$. From  $\neg \exists_{yy\are  \UInt zz} (F_{yy}x)$ we have (i): $\neg F_{zz^{*}}x$, because $zz^{*}\are \UInt zz$. From $F_{zz^{*}}a$ and \eqref{sum=sup}  we have $\supr_{zz^{*}}a$, thus $\forall_{z\prec zz^{*}}(\ing zx)\to \ing ax)$, and so $\ing ax$, from the definition of $zz^{*}$. Assume furthermore that $a\not =x$. Then from $\pr ax$ and \eqref{wsp} we have that exists $v$ such that $\ing vx\land \neg Ov a$. 
Now, we additionally indirectly assume that for some $t\prec zz$  we have $O v t$. So there is $u$ such that $\ing ut \land  \ing uv$. From $\ing uv$ and $\ing vx$ we have  $\ing ux$ by \eqref{transing}. Having that $\ing ut$, $t\prec zz$ and $\ing ux$, from definitions of $zz^{*}$ and \eqref{UP} we obtain $u\prec zz^{*}$ but then $\ing ua$ and this contradict to $\ing uv$ and earlier obtained $\neg O v a$. Thus we have $\forall_{t\prec zz} \neg Ov t$. From assumption we have $F_{zz}y$, so from definition \eqref{defsum} and transposition we have $\forall_{t\prec zz} (\neg O t v)\to  \neg \ing vy$, predecessor we already have thus $\neg \ing vy$. However, we have $\ing vx$, and $\ing xy$ from basic assumption, thus by \eqref{transing} $\ing vy$, which gives a contradiction. So, we have $x=a$ and so $\neg F_{zz^{*}}a$ by (i) but we know that $a$ is the sum of $zz^{*}a$ thus contradiction.
\end{proof}
\begin{lemma}\label{WSP+tosep}
In $\mathsf{GEM}^{\mathsf{pl}}_{P}+\eqref{UP}$ formula \eqref{WSP} is provable.
\end{lemma}
\begin{proof}
We assume $\s_{\I x \cup \I y}y \land x\not = y$ and $\forall z (z\prec \UInt \I y \to \exists u (\s_{\UInt \I x\cap \UInt \I z }u))$. From $\s_{\I x \cup \I y}y$ and definition of fusion \eqref{defsum} we obtain $Pxy$, thus using assumed $x\not = y$ we have $PPxy$. Using the latter and \eqref{wsp} we obtain $PP ay \land \neg O ax$ for some $a$. From $Pay$ and definition of \eqref{UP} we obtain $a\prec \UInt \I y$, thus from assumption taking $a/z$ we obtain $\s_{\UInt \I x\cap \UInt \I a }b$ for some $b$. The latter by definition of fusion \eqref{defsum} and \eqref{ring} implies that $\exists z (z \prec \UInt \I x\cap \UInt \I a)$. Thus by definition \eqref{UP} we have $\exists z (Pz x \land Pz a)$, i.e. $Oxa$ which gives a  contradiction.
\end{proof}
\begin{lemma}\label{lemextF}
In $\mathsf{GEM}^{\mathsf{pl}}_{P}$ formula \eqref{F=} is provable.
\end{lemma}
\begin{proof}
We assume indirectly that $F_{zz}x\land zz\approx yy$ and $\neg F_{yy}x$
We obtain (i): $\forall_{z\prec zz} Pz x$, (ii): $\forall y(P y x
\to \exists_{z\prec zz} O z y)$, and (iii): $\neg \forall_{z \prec
yy} Pz x \lor \neg \forall y(P y x \to \exists z (z\prec yy \land  O
z y ))$. If $\neg \forall_{z\prec yy} Pz x$ then we take
$a/z\colon a\prec yy \land \neg Pax$, so $a\prec zz$ by $zz\approx yy$
but then using (i) we have $Pax$ which is false. Therefore, we have
$\forall_{z\prec yy} Pz x$ so using (iii) we obtain $\neg \forall
y(P y x \to \exists z (z\prec yy \land  O z y ))$ we take $b/y\colon Pbx
\land \forall_{z\prec yy} \neg O z b$. From $Pbx$ and (ii) we have
that $\exists_{z\prec zz}  O z b$ so we take $c/z\colon$
$c\prec zz \land Ocb$. From  $c\prec zz$ and  $zz\approx yy$ we have
$c\prec yy$ thus using $\forall_{z\prec yy}\neg O z b$ and $c/z$ we obtain $\neg Ocb$
which gives a contradiction.
\end{proof}
So, all axioms of $\mathsf{GEM}^{\mathsf{pl}}_{F}$ are theorems of $\mathsf{GEM}^{\mathsf{pl}}_{P}$. Lastly, what we need to show is that the definition \eqref{dfing} in $\mathsf{GEM}^{\mathsf{pl}}_{F}$ is a thesis in  $\mathsf{GEM}^{\mathsf{pl}}_{P}$, definition of \eqref{defUcpn} in $\mathsf{GEM}^{\mathsf{pl}}_{F}$ is a thesis in $\mathsf{GEM}^{\mathsf{pl}}_{P}$ and conversely
that definition \eqref{UP} in $\mathsf{GEM}^{\mathsf{pl}}_{P}$ is thesis in $\mathsf{GEM}^{\mathsf{pl}}_{F}$. 
\begin{lemma}\label{defPF}
Definition \eqref{dfing} in $\mathsf{GEM}^{\mathsf{pl}}_{F}$ is a thesis in $\mathsf{GEM}^{\mathsf{pl}}_{P}$.
\end{lemma}
\begin{proof}
(``$\leftarrow$'') We assume $\exists zz (F_{zz}y\land x\prec zz)$. From the definition of fusion \eqref{defsum} we know that all $zz$ are parts of $y$ thus also $x$, i.e. $Pxy$. (``$\to$''). We assume indirectly that $Pxy$ and $\forall zz (F_{zz}y\to \neg x\prec zz)$. We take $\I x\cup \I y/zz$ and obtain $F_{\I x\cup \I y}y\to \neg x\prec \I x\cup \I y$. The right side of the latter implication is false, thus $\neg F_{\I x\cup \I y}y$. We have that $\forall_{z\prec \I x\cup \I y}Pzy$, thus from the definition of fusion \eqref{defsum} and $\forall_{z\prec \I x\cup \I y}Pzy$ we obtain that for some $a$ we obtain $Pay$ and $\neg O ay$ which is contradictory.
\end{proof}
As an immediate consequence of Lemma \ref{defPF}, we obtain 
\begin{lemma}\label{defUF}
Definition \eqref{defUcpn} in $\mathsf{GEM}^{\mathsf{pl}}_{F}$ is a thesis of $\mathsf{GEM}^{\mathsf{pl}}_{P}+\eqref{UP}$.
\end{lemma}
Lemmas \ref{lid}-\ref{defUF} give that 
\begin{theorem}\label{FsubP}
$\mathsf{GEM}^{\mathsf{pl}}_{F}+\eqref{dfing}\subseteq \mathsf{GEM}^{\mathsf{pl}}_{P}+\eqref{UP}$
\end{theorem}
Definition \eqref{UP} in $\mathsf{GEM}^{\mathsf{pl}}_{P}$ is thesis of  $\mathsf{GEM}^{\mathsf{pl}}_{F}$, and this is simple consequence of the definitions \eqref{dfing} and \eqref{defUcpn}.  Thus, using this, along with Theorems \ref{PsubF} and \ref{FsubP}, we ultimately arrive at the main result of the work. Our axiomatization of general extensional mereology with fusion as a primitive notion is definitionally equivalent to general extensional mereology with the primitive notion of being a part.
\begin{theorem}
$\mathsf{GEM}^{\mathsf{pl}}_{F}+\eqref{dfing}=\mathsf{GEM}^{\mathsf{pl}}_{P}+\eqref{UP}.$
\end{theorem}
\begin{small}

\end{small}

\begin{thebibliography}{99}
\bibitem{CoVa21} Cotnoir A. J. Varzi C. A. (2021) \textit{Mereology}, Oxford University Press
\bibitem{Ho08} Hovda, P. (2008) ``What is Classical Mereology?'', \textit{Journal of Philosophical Logic}, 38: 55–82
\bibitem{Fi10} Fine, K. (2010) ``Towards a Theory of Part'', \textit{Journal of Philosophy}, 107(11):  559–589
\bibitem{FlLi21} Florio, S., Linnebo, Ø. (2021) \textit{The Many and the One: A Philosophical Study of Plural Logic}, Oxford: Oxford University Press
 \bibitem{KeSh16} Ketland, J. and Schindler, T. (2016) ``Arithmetic with Fusions'', \textit{Logique et Analyse}, 59: 207–226
\bibitem{Kl17} Kleinschmidt, S. (2019) ``Fusion First'', \textit{Noûs}, 53(3): 689-707
\bibitem{Le92} Leśniewski, S. (1992) Collected Works I and II; Surma, S.J., Srzednicki, J.T., Barnett, D.I., Rickey, V.F., Eds.; Kluwer: Dordrecht, The Netherlands, 1992
\bibitem{Ly22} Łyczak, M. (2022)``Atomism Axiomatised Using Mereological Composition as a Primitive Notion'', 
\url{https://doi.org/10.48550/arXiv.2310.15087} 
\bibitem{MaAr22} Manzano, M., Aranda, W. (2022) ``Many-Sorted Logic'', \textit{The Stanford Encyclopedia of Philosophy}, \url{https://plato.stanford.edu/entries/logic-many-sorted/#SecoOrdeLogi}
\bibitem{Ma05} Manzano, M. (2005) \textit{Extensions of First-Order Logic}, Cambridge Tracts in Theoretical Computer Sci-ence, Vol. 19, Cambridge: Cambridge University Press
\bibitem{Si87} Simons, P. M. (1987) \textit{Parts: A Study in Ontology}, Oxford: Clarendon Press.
\bibitem{Sw} Świętorzecka, K., Łyczak, M. (2020) ``Mereology with Super-Supplementation Axioms'',
\textit{Logic and Logical Philosophy}, 29(2): 189–211
\bibitem{Sb84} Sobociński, B. (1984) ``Studies in Leśniewski’s Mereology'', In: Srzednicki, J.T.J.,
Rickey, V.F. (eds) Leśniewski’s Systems. Nijhoff International Philosophy Series,
vol. 13, Dordrecht: Springer
\bibitem{Va16}  Varzi, A. (2019) ``Mereology'', \textit{The Stanford Encyclopedia of Philosophy} (Spring 2019 Edition), E. Zalta (ed.),
\url{https://plato.stanford.edu/archives/spr2019/entries/mereology/}
\end{thebibliography}
\end{document}